\documentclass[11pt,letterpaper]{article}
\usepackage[margin=1in]{geometry}

\pdfoutput=1

\usepackage{amsmath}
\usepackage{amsfonts}
\usepackage{graphicx}
\usepackage{amsthm}
\usepackage{color}
\usepackage[usenames,dvipsnames]{xcolor}
\usepackage{hyperref}
\hypersetup{colorlinks=true,
    linkcolor=BrickRed,
    citecolor=OliveGreen,
    urlcolor=MidnightBlue}

  \theoremstyle{plain}
  \newtheorem{proposition}{Proposition}
  \newtheorem{theorem}{Theorem}
  \newtheorem{lemma}{Lemma}
  \theoremstyle{definition}
  
  \newtheorem{example}{Example}
\newtheorem{propositiononedual}{Proposition}{\bf}{\it}
\newtheorem{propositiontwodual}{Proposition}{\bf}{\it}
\newtheorem{propositionthreedual}{Proposition}{\bf}{\it}

\newcommand{\R}{\mathbb{R}}

\newcommand{\Sym}{\mathbb{S}}
\newcommand{\orthant}[2]{\R^{#1,(#2)}_+}

\newcommand{\psdcone}[2]{\Sym^{#1,(#2)}_+}

\newcommand{\soc}{Q}
\newcommand{\E}{\mathcal{E}}
\newcommand{\sumk}{s}
\newcommand{\SH}{\textup{SH}}
\newcommand{\psd}{\succeq}
\newcommand{\pd}{\succ}
\newcommand{\nsd}{\preceq}

\newcommand{\elem}{e}
\newcommand{\melem}{E}
\newcommand{\tr}{\textup{tr}}
\newcommand{\ones}{\mathbf{1}}
\DeclareMathOperator*{\diag}{\textup{diag}}
\newcommand{\Amap}{A}
\newcommand{\Bmap}{B}

\title{Polynomial-sized Semidefinite Representations of Derivative Relaxations of Spectrahedral Cones}
\author{James Saunderson \and Pablo A.~Parrilo\thanks{The authors are with the 
        Laboratory for Information and Decision Systems, Department of Electrical Engineering and Computer Science, 
    Massachusetts Institute of Technology, Cambridge MA 02139, USA\@. Email: \{jamess,parrilo\}@mit.edu.
This research was funded by the Air Force Office of Scientific Research under grants 
FA9550-11-1-0305 and FA9550-12-1-0287.}
}

\begin{document}
    \maketitle
    \begin{abstract}
        We give explicit polynomial-sized (in $n$ and $k$) semidefinite
        representations of the hyperbolicity cones associated with the
        elementary symmetric polynomials of degree $k$ in $n$ variables.  These
        convex cones form a family of non-polyhedral outer approximations of
        the non-negative orthant that preserve low-dimensional faces while
        successively discarding high-dimensional faces. 
        More generally we construct explicit semidefinite representations
        (polynomial-sized in $k,m$, and $n$) of the hyperbolicity cones
        associated with $k$th directional derivatives of polynomials of the
        form $p(x) = \det(\sum_{i=1}^{n}A_i x_i)$ where the $A_i$ are $m\times
        m$ symmetric matrices.  These convex cones form an analogous family of
        outer approximations to any spectrahedral cone. 
        Our representations allow us to use semidefinite programming to
        solve the linear cone programs associated with these convex
        cones as well as their (less well understood) dual cones.
    
        % \keywords{hyperbolic polynomial \and hyperbolicity cone \and elementary
        % symmetric polynomial \and semidefinite representation}
        % \subclass{90C22 \and % semidefinite programming;
        %           90C25 \and % convex programming;
        %           52A41 \and % convex functions and convex programs
        %           52A20}     % convex sets in n dimensions
        
\end{abstract}

    \section{Introduction}
   \label{sec:intro}
    Expressing convex optimization problems in conic form, as the minimization
    of a linear functional over an affine slice of a convex cone, has been an
    important method in the development of modern convex optimization theory.
    This abstraction is useful (at least from a theoretical viewpoint) because
    all that is difficult and interesting about the problem is packaged into
    the cone.  The conic viewpoint provides a natural way to organize classes
    of convex optimization problems into hierarchies based on whether the cones
    associated with one class can be expressed in terms of the cones associated
    with another class.  For example, semidefinite programming generalizes
    linear programming because the non-negative orthant is the restriction to
    the diagonal of the positive semidefinite cone.

    When faced with a convex cone the geometry of which is not well understood,
    we stand to gain theoretical insight as well as off-the-shelf optimization
    algorithms by representing it in terms of a cone with known geometric and
    algebraic structure such as the positive semidefinite cone. Terminology is
    attached to this idea, with a cone being \emph{spectrahedral} if it is a
    linear section (or `slice') of the positive semidefinite cone, and \emph{semidefinitely
    representable} if it is a linear projection of a spectrahedral cone. The
    efficiency of a semidefinite representation is also clearly important. If
    we can write a cone as the projection of a slice of the cone of $m\times m$
    positive semidefinite matrices, we say it has a semidefinite representation
    of \emph{size $m$}.  Many convex cones have been shown to be semidefinitely
    representable using a variety of techniques (see \cite{nemirovski2001lectures} as 
    well as the recent book \cite{frgbook} for contrasting methods and examples). 
    
    The classes of semidefinitely representable cones and spectrahedral cones
    are distinct \cite{ramana1995some}, with semidefinitely representable cones
    being perhaps more natural from the point of view of optimization.  A
    semidefinite representation of a cone suffices to express the associated
    cone program as a semidefinite program.  Furthermore, unlike spectrahedral
    cones, the class of semidefinitely representable cones is closed under
    duality~\cite[Proposition 3.2]{gouveia2011positive}.
    
    The \emph{hyperbolicity cones} form a family of convex cones (constructed
    from certain multivariate polynomials) that includes the positive
    semidefinite cone, as well as all homogeneous
    cones~\cite{guler1997hyperbolic}.  While it has been shown (by Lewis et
    al.~\cite{lewis2005lax} based on work of Helton and
    Vinnikov~\cite{helton2007linear}) that all three-dimensional hyperbolicity
    cones are spectrahedral, little is known about semidefinite representations
    of higher dimensional hyperbolicity cones. Furthermore while hyperbolicity
    cones have very simple descriptions, their dual cones are not well
    understood. 
    
    In this paper we give explicit, polynomial-sized semidefinite
    representations of the hyperbolicity cones known as the \emph{derivative
        relaxations} of the non-negative orthant, and the corresponding
    derivative relaxations of the positive semidefinite cone. These cones form
    a family of outer approximations to the orthant and positive semidefinite
    cones respectively with many interesting
    properties~\cite{renegar2006hyperbolic}. We obtain semidefinite
    representations of the derivative relaxations of spectrahedral cones as
    slices of the derivative relaxations of the positive semidefinite cone.

    \subsection{Hyperbolic polynomials and hyperbolicity cones}
    \label{sec:hyp-poly}
    A homogeneous polynomial $p$ of degree $m$ in $n$ variables is
    \emph{hyperbolic} with respect to $e\in \R^n$ if $p(e) \neq 0$ and if for
    all $x\in \R^n$ the univariate polynomial $t\mapsto p(x-te)$ has only real
    roots. G\r{a}rding's foundational work on hyperbolic polynomials
    \cite{garding1959inequality} establishes that if $p$ is hyperbolic with
    respect to $e$ then the connected component of $\{x\in \R^n: p(x) \neq 0\}$
    containing $e$ is an open convex cone. This cone is called the
    \emph{hyperbolicity cone} corresponding to $(p,e)$. We denote it by
    $\Lambda_{++}(p,e)$, and its closure by $\Lambda_{+}(p,e)$.
   
    Note that $p$ is hyperbolic with respect to $e$ if and only if $-p$ is
    hyperbolic with respect to $e$. As such we assume throughout that $p(e) >
    0$.  We can expand $p(x+te)$ as 
    \[ p(x+te) = p(e)\left[t^{m} + a_1(x)t^{m-1} + a_2(x)t^{m-2} +
        \cdots + a_{m-1}(x)t + a_m(x)\right]\]
    where the $a_i(x)$ are polynomials that are homogeneous of degree $i$.  There
    is an alternative description of the hyperbolicity cone $\Lambda_{+}(p,e)$ due
    to Renegar~\cite[Theorem 20]{renegar2006hyperbolic} as
    \begin{equation}
        \label{eq:coefcone}
        \Lambda_{+}(p,e) = \left\{x\in \R^n: a_1(x) \geq 0,\;\;a_2(x) \geq 0,\;\; 
            \ldots,\;\;a_m(x) \geq 0\right\}.
    \end{equation}
    We use this description of $\Lambda_+(p,e)$ throughout the paper.    
   \paragraph{Basic examples:}
\begin{itemize}
    \item The polynomial $p(x_1,x_2,\ldots,x_n) = x_1 x_2 \cdots x_n$ is
        hyperbolic with respect to $e=\ones_n:=(1,1,\ldots,1)$.  The associated
        closed hyperbolicity cone is the non-negative orthant, $\R_+^n$. Since
        \[ p(x+t\ones_n) = t^{n} + \elem_1(x)t^{n-1} + \cdots + \elem_{n-1}(x)t + \elem_n(x)\]
        where  $\elem_k(x) = \sum_{1\leq i_1<\cdots<i_k \leq n} x_{i_1}\cdots
        x_{i_k}$ is the elementary symmetric polynomial of degree $k$ in the
        variables $x_1,x_2,\ldots,x_n$, 
        \[ \Lambda_{+}(p,e) = \R^{n}_+ = \left\{x\in \R^n: \elem_1(x) \geq 0,\;\;
                \elem_2(x)\geq 0,\;\;\ldots,\;\;\elem_n(x)\geq 0\right\}.\]
    \item Let $X$ be an $n\times n$ symmetric matrix of indeterminates.  The
        polynomial $p(X) = \det(X)$ is hyperbolic with respect to $e=I_n$, the
        $n\times n$ identity matrix.  The associated closed hyperbolicity cone
        is the positive semidefinite cone, $\Sym_+^n$.  Since
        \[ p(X+tI_n) = t^{n} + \melem_{1}(X)t^{n-1} + \cdots + \melem_{n-1}(X)t + \melem_n(X)\]
        where the $\melem_k(X)$ are the coefficients of the characteristic polynomial of $X$, 
        \[ \Lambda_+(p,e) = \Sym_+^n = \left\{X: \melem_1(X)\geq 0,\;\;\melem_2(X)\geq 0,
                \;\;\ldots,\;\;\melem_n(X) \geq 0\right\}.\]
        Observe that $\melem_{k}(X):= \elem_k(\lambda(X))$ is the elementary
        symmetric polynomial of degree $k$ in the eigenvalues of $X$ so the
        positive semidefinite cone can also be described in terms of polynomial
        inequalities on the eigenvalues of $X$ as
        \[ \Sym_+^n = \left\{X: \elem_1(\lambda(X)) \geq 0,\;\;\elem_2(\lambda(X)) \geq 0,\;\;
                \ldots,\;\;\elem_n(\lambda(X)) \geq 0\right\}.\]
\end{itemize}

\subsection{Derivative relaxations}
\label{sec:deriv-relaxations}
If $p$ is hyperbolic with respect to $e$ then (essentially by Rolle's theorem)
the directional derivative of $p$ in the direction $e$, \emph{viz.} 
\[    p^{(1)}_e(x):= \left.\frac{d}{dt}p(x+te)\right|_{t=0}\] 
is also hyperbolic with respect to $e$, a construction that goes back to
G\r{a}rding \cite{garding1959inequality}. If $p$ has degree $m$, by repeatedly
differentiating in the direction $e$ we construct a sequence of polynomials
$p,p^{(1)}_e,p^{(2)}_e,\ldots,p^{(m-1)}_e$ each hyperbolic with respect to $e$.

The corresponding hyperbolicity cones can be expressed nicely in terms of
polynomial inequalities. Indeed if $p(x+te) = p(e)\left[t^m + \sum_{i=1}^{m}
    a_{i}(x)t^{m-i}\right]$ then differentiating $k$ times with respect to $t$
we see that 
\[ p_e^{(k)}(x+te) = p(e)\left[c_0a_{m-k}(x) + c_1a_{m-k-1}(x)t + 
        \cdots +c_{m-k}t^{m-k}\right]\] 
where $c_i = (k+i){!}/i{!} > 0$. By~\eqref{eq:coefcone} the corresponding
hyperbolicity cone is
\[\Lambda_+^{(k)}(p,e):=\Lambda_{+}(p^{(k)}_e,e) = 
    \{x\in \R^n:a_1(x)\geq 0,\;\; a_2(x)\geq 0,\;\;\ldots, \;\;a_{m-k}(x) \geq 0\}\]
and can be obtained from \eqref{eq:coefcone} by removing $k$ of the inequality
constraints. As a result, the hyperbolicity cones $\Lambda_+^{(k)}(p,e)$
provide a sequence of outer approximations to the original hyperbolicity cone
that satisfy 
\[  \Lambda_{+}(p,e) \subset \Lambda_{+}^{(1)}(p,e) \subset \cdots \subset
    \Lambda_{+}^{(m-1)}(p,e).\] 
The last of these, $\Lambda_{+}^{(m-1)}(p,e)$, is simply the closed half-space
defined by $e$.  The work of Renegar \cite{renegar2006hyperbolic} highlights
the many nice properties of this sequence of approximations.

Note that we abuse terminology by referring to the cones
$\Lambda_{+}^{(k)}(p,e)$ as  \emph{derivative relaxations} of the hyperbolicity
cone $\Lambda_{+}(p,e)$. The abuse is that $\Lambda_{+}^{(k)}(p,e)$ does not
depend only on the \emph{geometric} object $\Lambda_+(p,e)$ but on its
particular \emph{algebraic} description via $p$ and $e$. 

\paragraph{Examples:}
\begin{itemize}
    \item In the case of $p(x) = x_1x_2\cdots x_n = \elem_n(x)$ and
        $e=\ones_n$, we have that $p^{(k)}_e(x) = k{!}\elem_{n-k}(x)$.
        Consequently the $k$th derivative relaxation of the orthant, which we
        denote by $\orthant{n}{k}$, is the hyperbolicity cone
        $\Lambda_{+}(\elem_{n-k},\ones_n)$. It can be expressed as
\begin{align}
    \orthant{n}{k} & = \{x\in \R^n: 
        \elem_1(x) \geq 0,\;\; \elem_2(x) \geq 0,\;\; \ldots,
        \;\; \elem_{n-k}(x) \geq 0\}.\label{eq:coef-deriv}
\end{align}
Consistent with these descriptions we define $\orthant{n}{n} := \R^n$.
    \item In the case of $p(X) = \det(X) = \melem_n(X)$ and $e=I_n$, we have
        that $p^{(k)}_e(x) = k{!}\melem_{n-k}(X)$.  The $k$th derivative
        relaxation of the positive semidefinite cone, which we denote by
        $\psdcone{n}{k}$, can be described as
\begin{align}
    \psdcone{n}{k} & = \left\{X\in \Sym^n: 
        \melem_1(x) \geq 0,\;\; \melem_2(x) \geq 0,\;\; \ldots,\;\; 
        \melem_{n-k}(x) \geq 0\right\}\label{eq:psd-coef-deriv}\\
    & = \left\{X\in \Sym^n: \elem_1(\lambda(X)) \geq 0,\;\; 
        \elem_2(\lambda(X))\geq 0,\;\;\ldots,\;\;
        \elem_{n-k}(\lambda(X)) \geq 0\right\}.\label{eq:psd-coef-deriv-spec}
\end{align}
Again we define $\psdcone{n}{n} := \Sym^n$, the set of $n\times n$ symmetric
matrices.  Since $\melem_i(\diag(x)) = \elem_i(x)$ for all $i$, the diagonal
slice of $\psdcone{n}{k}$ is exactly $\orthant{n}{k}$.
\end{itemize}

\paragraph{Symmetry:}
Suppose $G$ is a group acting by linear transformations on $\R^n$ by $x\mapsto
g\cdot x$ for all $g\in G$.  Suppose \emph{both $p$ and $e$} are invariant
under the group action, i.e., $g\cdot e = e$ and $(g\cdot p)(x) :=
p(g^{-1}\cdot x) = p(x)$ for all $g\in G$. Then for all $t\in \R$, $x\in\R^n$
and $g\in G$
\[ p(x+te) = (g\cdot p)(x+te) = p(g^{-1}\cdot (x+te)) = p((g^{-1}\cdot x) + te).\]
Hence the hyperbolicity cone $\Lambda_+(p,e)$ and all of its derivative cones 
$\Lambda_+^{(k)}(p,e)$ are invariant under this same group action. 

For our purposes an important example of this is the symmetry of the cones
$\psdcone{n}{k}$.  The action of $O(n)$ by conjugation on symmetric matrices
leaves the polynomial $p(X) = \det(X)$ invariant \emph{and} preserves the
direction $e = I_n$. Hence all of the derivative relaxations of the positive
semidefinite cone are invariant under conjugation by orthogonal matrices.  As
such, the cones $\psdcone{n}{k}$ are \emph{spectral sets}, in the sense that whether a
symmetric matrix $X$ belongs to $\psdcone{n}{k}$ depends only on the
eigenvalues of $X$. This is evident from the description of $\psdcone{n}{k}$
in~\eqref{eq:psd-coef-deriv-spec}.

\subsection{Related work}
\label{sec:related}
Previous work has focused on semidefinite and spectrahedral representations of
the derivative relaxations of the orthant.  Zinchenko
\cite{zinchenko2008hyperbolicity} used a decomposition approach to give
semidefinite representations of $\orthant{n}{1}$ and its dual cone. Sanyal
\cite{sanyal2013derivative} subsequently gave spectrahedral representations of
$\orthant{n}{1}$ and $\orthant{n}{n-2}$ and conjectured that all of the
derivative relaxations of the orthant admit spectrahedral representations.

Recently Br{\"a}nd{\'e}n \cite{branden2014hyperbolicity} settled this
conjecture in the affirmative giving spectrahedral representations of
$\orthant{n}{n-k}$ for $k=1,2,\ldots,n-1$ of size $O(n^{k-1})$. For each $1\leq
k < n$  Br{\"a}nd{\'e}n constructs a graph $G_{n,k} = (V,E)$ together with edge
weights $(w_e(x))_{e\in E}$ that are linear forms in $x$ so that 
\begin{equation}
    \label{eq:branden}
    \orthant{n}{n-k} = \left\{x\in \R^n: L_{G_{n,k}}(x) \psd 0\right\}
\end{equation}
where $L_{G_{n,k}}(x)$ is the $|V|\times |V|$ edge-weighted Laplacian of
$G_{n,k}$.  Since $L_{G_{n,k}}(x)$ is linear in the edge weights, and the edge
weights are linear forms in $x$, \eqref{eq:branden} is a spectrahedral
representation of size $|V|$.  With the exception of two distinguished
vertices, the vertices of $G_{n,k}$ are indexed by all $\ell$-tuples (for
$1\leq \ell\leq k-1$) consisting of distinct elements of $\{1,2,\ldots,n\}$.
Hence $|V| = 2+\sum_{\ell=1}^{k-1}\ell{!}\binom{n}{\ell}$ showing that
Br{\"a}nd{\'e}n's spectrahedral representation of $\orthant{n}{n-k}$ has size
$O(n^{k-1})$. While Br{\"a}nd{\'e}n's construction is of considerable
theoretical interest, these representations (unlike ours) are not practical for
optimization due to their prohibitive size.

A spectrahedral representation of $\orthant{n}{1}$ is implicit in the work of
Choe et al.~\cite{choe2004homogeneous} that studies the relationships between
matroids and hyperbolic polynomials.  Choe et al.~observe that if $\mathcal{M}$
is a \emph{regular matroid} represented by the rows of a totally unimodular
matrix $V$ then $\det(V^T\diag(x)V)$ is the basis generating polynomial of
$\mathcal{M}$.  In particular, the uniform matroid $U_{n}^{n-1}$ is regular and
has $\elem_{n-1}(x)$ as its basis generating polynomial, yielding a symmetric
determinantal representation of $\elem_{n-1}(x)$ and hence a spectrahedral
representation of $\orthant{n}{n-1}$. 

From a computational perspective, G\"{u}ler \cite{guler1997hyperbolic} showed
that if $p$ has degree $m$ and is hyperbolic with respect to $e$ then $\log p$
is a self-concordant barrier function (with barrier parameter $m$) for the
hyperbolicity cone $\Lambda_+(p,e)$.  As such, as long as $p$ and its gradient
and Hessian can be computed efficiently, one can use interior point methods to
minimize a linear functional over an affine slice of $\Lambda_+(p,e)$
efficiently. Renegar \cite[Section 9]{renegar2006hyperbolic} gave an efficient
interpolation-based method for computing $p_e^{(k)}$ (and its gradient and
Hessian) whenever $p$ (and its gradient and Hessian) can be evaluated
efficiently. G\"{u}ler and Renegar's observations together yield efficient
computational methods to optimize a linear functional over an affine slice of a
derivative relaxation of a spectrahedral cone. Our results complement these,
giving a method to solve optimization problems of this type using existing
numerical procedures for semidefinite programming.

\subsection{Notation}

Here we define notation not explicitly defined elsewhere in the paper.  If $C$
is a convex cone, we denote by $C^{*}$ the dual cone, i.e.~the set of linear
functionals that are non-negative on $C$. We represent linear functionals on
$\R^n$ using the standard Euclidean inner product, and linear functionals on
$\Sym^n$ using the trace inner product $\langle X,Y\rangle = \tr(XY)$.  As such
$C^* = \{y: \langle y,x\rangle \geq 0,\;\;\text{for all $x\in C$}\}$.  If $X\in
\Sym^n$ let $\lambda(X)$ denote its eigenvalues sorted so that
$\lambda_1(X)\geq \lambda_2(X)\geq \cdots \geq \lambda_n(X)$.  If $X\in\Sym^n$
let $\diag(X)\in \R^n$ denote the vector of diagonal entries and if $x\in \R^n$
let $\diag(x)$ denote the diagonal matrix with diagonal entries given by $x$.
The usage will be clear from the context.

\section{Results}
\label{sec:results}
Our main contribution is to construct two different explicit polynomial-sized
semidefinite representations of the derivative relaxations of the positive
semidefinite cone.  We call our two representations the \emph{derivative-based}
and \emph{polar derivative-based} representations respectively.  In this section
we describe these representations, and outline the proof of our main
theoretical result.
\begin{theorem}
    \label{thm:main}
    For each positive integer $n$ and each $k=1,2,\ldots,n-1$, the cone
    $\psdcone{n}{k}$ has a semidefinite representation of size
    $O(\min\{k,n-k\}n^2)$.
\end{theorem}
We defer detailed proofs of the correctness of our representations to
Sections~\ref{sec:main-pfs} and~\ref{sec:btn}.  At this stage, we just
highlight that there is essentially one basic algebraic fact that underlies all
of our results. Whenever $V_n$ is an $n\times (n-1)$ matrix with orthonormal
columns that are each orthogonal to $\ones_n$, i.e.\ $V_n^TV_n = I_{n-1}$ and
$V_n^T\ones_n = 0$, then
    \[ \elem_{n-1}(x) = n\det(V_n^T\diag(x)V_n).\]
    We give a proof of this identity in Section~\ref{sec:main-pfs}. Note that
    this identity is independent of the particular choice of $V_n$ satisfying
    $V_n^TV_n = I_{n-1}$ and $V_n^T\ones_n = 0$. In fact, all of the results
    expressed in terms of $V_n$ (notably
    Propositions~\ref{prop:RS1},~\ref{prop:RS2},~\ref{prop:RS1-dual},
    and~\ref{prop:RS2-dual}) are similarly independent of the particular choice
    of $V_n$. 

Both of the representations are recursive in nature. The derivative-based
representation is based on recursively applying two basic propositions
(Propositions~\ref{prop:btn} and~\ref{prop:RS1}, to follow) to construct a
chain of semidefinite representations of the form 
\begin{align}
    \framebox{$\psdcone{n}{k}$} \xleftarrow[\textup{Prop.~\ref{prop:btn}}]{O(n^2)} \orthant{n}{k} \xleftarrow[\textup{Prop.~\ref{prop:RS1}}]{0}
    \framebox{$\psdcone{n-1}{k-1}$}\xleftarrow[\textup{Prop.~\ref{prop:btn}}]{O((n-1)^2)}& \orthant{n-1}{k-1} \leftarrow \cdots\label{eq:td} \\
    \cdots&\leftarrow \orthant{n-k+1}{1} \xleftarrow[\textup{Prop.~\ref{prop:RS1}}]{0} \framebox{$\psdcone{n-k}{0}$.}\nonumber
\end{align}
The annotated arrow $C\xleftarrow[\textup{Prop.~$a$}]{m} K$ indicates that given
a semidefinite representation of $K$ of size $m'$ we can construct a
semidefinite representation of $C$ of size $m'+m$, and that an explicit
description of the construction is given in Proposition $a$.

The base case of the recursion is just the positive semidefinite cone
$\psdcone{n-k}{0}$, which has a trivial semidefinite representation. Hence
starting from $\psdcone{n-k}{0}$ (which has a semidefinite representation of
size $n-k$), we can apply Proposition~\ref{prop:RS1} to obtain a semidefinite
representation of $\orthant{n-k+1}{1}$ of size $n-k$, then apply
Proposition~\ref{prop:btn} to obtain a semidefinite representation of
$\psdcone{n-k+1}{1}$ of size $(n-k) + O((n-k+1)^2)$, and so on.

The polar derivative-based representation is based on recursively applying
Proposition~\ref{prop:btn} together with a third basic proposition
(Proposition~\ref{prop:RS2}, to follow) to construct a slightly different chain
of semidefinite representations of the form
\begin{align}
    \framebox{$\psdcone{n}{k}$} \xleftarrow[\textup{Prop.~\ref{prop:btn}}]{O(n^2)} \orthant{n}{k} \xleftarrow[\textup{Prop.~\ref{prop:RS2}}]{n}
    \framebox{$\psdcone{n-1}{k}$} \xleftarrow[\textup{Prop.~\ref{prop:btn}}]{O(n^2)} &\orthant{n-1}{k} \leftarrow \cdots\nonumber\\
    \cdots & \leftarrow \orthant{k+2}{k} \xleftarrow[\textup{Prop.~\ref{prop:RS2}}]{n} \framebox{$\psdcone{k+1}{k}$.}\label{eq:bu}
\end{align}
Note that the base case of the recursion is just $\psdcone{k+1}{k} = \{X\in
    \Sym^{k+1}:\; \tr(X) \geq 0\}$, a half-space.

\subsection{Building blocks of the two recursions}

We now describe the constructions related to each of the types of arrows in the
recursions sketched above. The arrows labeled by Proposition~\ref{prop:btn}
assert that we can construct a semidefinite representation of $\psdcone{n}{k}$
from a semidefinite representation of $\orthant{n}{k}$. This can be done in the
following way. 
\begin{proposition}
    \label{prop:btn}
    If $\orthant{n}{k}$ has a semidefinite representation of size $m$, then
    $\psdcone{n}{k}$ has a semidefinite representation of size $m+O(n^2)$.
    Indeed 
    \begin{equation}
        \label{eq:whysh} \psdcone{n}{k} = \left\{X\in \Sym^n: 
            \exists z\in \R^n\;\;\text{s.t.}\;\;
            z\in \orthant{n}{k},\;\;(X,z) \in \SH_n\right\},
    \end{equation}
    where $\SH_n$ is the \emph{Schur-Horn cone} defined as
    \[ \SH_n = \left\{(X,z):\;\;z_1\geq z_2\geq \cdots \geq z_n,\;\;
            X\in \textup{conv}_{Q\in O(n)} \{Q^T\diag(z)Q\}\right\}\]
    i.e.~the set of pairs $(X,z)$ such that $X$ is in the convex hull
    of all symmetric matrices with ordered spectrum $z$. 
    The Schur-Horn cone has the semidefinite characterization 
    \begin{align}
        (X,z)\in\SH_n\quad\text{if and only if} &\quad z_1\geq z_2 \geq \cdots \geq z_n\;\;\text{and}\nonumber\\
        \text{there exist} & \quad t_2,\ldots,t_{n-1}\in \R,\;\;Z_2,\ldots,Z_{n-1} \psd 0\;\;\nonumber\\
        \text{such that}&\quad \tr(X) = \textstyle{\sum_{j=1}^{n}z_j},\;\;X \nsd z_1 I,\;\;\text{and}\nonumber\\
        \text{for $\ell=2,\ldots,n-1$,} & \quad X  \nsd t_{\ell} I + Z_{\ell}\;\;\text{and}\;\; \ell\cdot t_{\ell} + \tr(Z_{\ell})  
        \leq \textstyle{\sum_{j=1}^{\ell}z_j}.\nonumber
    \end{align}
\end{proposition}
Proposition~\ref{prop:btn} holds because of the \emph{symmetry} of
$\psdcone{n}{k}$. In particular it is a spectral set---invariant under conjugation by orthogonal
matrices. The other reason this representation works is that the diagonal slice of
$\psdcone{n}{k}$ is $\orthant{n}{k}$. We discuss this result in more detail in
Section~\ref{sec:btn}.

The arrows in~\eqref{eq:td} labeled by Proposition~\ref{prop:RS1} appear only
in the derivative-based recursion.  They assert that we can obtain a
semidefinite representation of $\orthant{n}{k}$ from a semidefinite
representation of $\psdcone{n-1}{k-1}$.  Indeed we establish in
Section~\ref{sec:RS1-pfs} that $\orthant{n}{k}$ is actually a slice of
$\psdcone{n-1}{k-1}$.
\begin{proposition}
    \label{prop:RS1}
    If $1\leq k \leq n-1$ then $\orthant{n}{k} = \left\{x\in \R^n:
        V_n^T\diag(x)V_n\in \psdcone{n-1}{k-1}\right\}$.
\end{proposition}

The arrows in~\eqref{eq:bu} labeled by Proposition~\ref{prop:RS2} appear only
in the polar derivative-based recursion. They assert that we can obtain a
semidefinite representation of $\orthant{n}{k}$ from a semidefinite
representation of $\psdcone{n-1}{k}$.  We establish the following in
Section~\ref{sec:RS2-pfs}.
\begin{proposition}
    \label{prop:RS2}
    If $1\leq k \leq n-2$ then 
    \[ \orthant{n}{k} = \left\{x\in \R^n: \exists Z\in \psdcone{n-1}{k}\;\;\text{s.t.}\;\;
            \diag(x) \psd V_n Z V_n^T\right\}.\]
\end{proposition}

\subsection{Size of the representations}

Recall that each arrow $C\xleftarrow{m} K$ in~\eqref{eq:td} and~\eqref{eq:bu}
is labeled with the \emph{additional size} $m$ required to implement the
representation of $C$ given a semidefinite representation of $K$. Since the
derivative-based recursion has $2k$ arrows, it is immediate from \eqref{eq:td}
that the derivative-based semidefinite representation of $\psdcone{n}{k}$ has
size $O(kn^2)$ and so is of polynomial size.

On the other hand, this approach gives a disappointingly large semidefinite
representation of the half-space $\psdcone{n}{n-1} = \{X\in \Sym^n: \tr(X) \geq
    0\}$ of size $O(n^3)$.  The derivative-based approach cannot exploit the
fact that this is a very simple cone.  This is why we also consider the polar
derivative-based representation, as it is designed around the fact that
$\psdcone{n}{n-1}$ has a simple semidefinite representation.

It is immediate from \eqref{eq:bu} that the polar derivative-based semidefinite
representation of $\psdcone{n}{k}$ has size $O((n-k)n^2)$ and so is also of
polynomial size. Furthermore, it gives small representations of size $O(n^2)$
exactly when the derivative-based representations are large, of size $O(n^3)$.
For any given pair $(n,k)$ we should always use the derivative-based
representation of $\psdcone{n}{k}$ if $k < n/2$ and the polar derivative-based
representation when $k> n/2$. Theorem~\ref{thm:main} combines our two size
estimates, stating that $\psdcone{n}{k}$ has a semidefinite representation of
size $O(\min\{k,n-k\}n^2)$. 

\subsection{Pseudocode for our derivative-based representation}
We do not write out any of our semidefinite representations in full because the
recursive descriptions given here are actually more naturally suited to
implementation. To illustrate this, we give pseudocode for the MATLAB-based
high-level modeling language YALMIP~\cite{lofberg2004yalmip} that `implements'
the derivative-based representations of $\psdcone{n}{k}$ and $\orthant{n}{k}$.
Decision variables are declared by expressions like \texttt{x = sdpvar(n,1);}
which creates a decision variable \texttt{x} taking values in $\R^n$. An LMI
object is a list of equality constraints and linear matrix inequality
constraints that are linear in any declared decision variables. 

Suppose we have a function \texttt{SH(X,z)} that takes a pair of decision
variables and returns an LMI object corresponding to the constraint that
$(X,z)\in \SH_n$. This is easy to construct from the explicit semidefinite
representation in Proposition~\ref{prop:btn}. Then the function
\texttt{psdcone} takes an $n\times n$ symmetric matrix-valued decision variable
\texttt{X} and returns an LMI object for the constraint $X\in \psdcone{n}{k}$.
\begin{align*}
    \texttt{1:}\quad&\texttt{function K = \underline{psdcone}(X,k)}\\
    \texttt{2:}\quad &\qquad\texttt{if k==0}\\
   \texttt{3:}\quad &\qquad\qquad\texttt{K = [X >= 0];}\\
   \texttt{4:}\quad &\qquad\texttt{else}\\
   \texttt{5:}\quad &\qquad\qquad\texttt{z = sdpvar(size(X,1),1);}\\
   \texttt{6:}\quad & \qquad\qquad\texttt{K = [\underline{orthant}(z,k), \underline{SH}(X,z)];}\\
   \texttt{7:}\quad &\qquad\texttt{end}\\\intertext{It calls a function \texttt{orthant} that takes a 
       decision variable \texttt{x} in $\R^n$ and returns an LMI object for the constraint $x\in \orthant{n}{k}$.}
    \texttt{1:}\quad&\texttt{function K = \underline{orthant}(x,k)}\\
    \texttt{2:}\quad&\qquad\texttt{if k==0}\\
    \texttt{3:}\quad&\qquad\qquad\texttt{K = [x >= 0];}\\
    \texttt{4:}\quad& \qquad\texttt{else}\\
    \texttt{5:}\quad& \qquad\qquad\texttt{V = null(ones(size(x))');}\\
   \texttt{6:}\quad &\qquad\qquad\texttt{K = [\underline{psdcone}(V'*diag(x)*V,k-1)];}\\
   \texttt{7:}\quad & \qquad\texttt{end}
\end{align*}
It is straightforward to adapt these two functions for the polar
derivative-based representation, one needs only to change the base cases (lines
2--4 of each) and to adapt line 6 of \texttt{orthant} to reflect
Proposition~\ref{prop:RS2}.

\subsection{Dual cones}
If a cone is semidefinitely representable, so is its dual cone. In fact there
are explicit procedures to take a semidefinite representation for a cone and
produce a semidefinite representation for its dual cone~\cite[Section
4.1.1]{nemirovski2006advances}.  Here we describe two explicit semidefinite
representations of the dual cones $(\psdcone{n}{k})^*$ that enjoy the same
recursive structure as the corresponding semidefinite representations of
$\psdcone{n}{k}$. 

To construct them, we essentially dualize all the relationships given by the
arrows in~\eqref{eq:td} and~\eqref{eq:bu}.  By straightforward applications of
a conic duality argument, in Section~\ref{sec:duals} we establish the following
dual analogues of Propositions~\ref{prop:RS1} and~\ref{prop:RS2}.
\begin{propositiontwodual}
    \label{prop:RS1-dual}
    If $1\leq k \leq n-1$ then 
    \[ (\orthant{n}{k})^{*} = \left\{\diag(V_n Y V_n^T):\;\; Y\in (\psdcone{n-1}{k-1})^*\right\}.\]
\end{propositiontwodual}
\begin{propositionthreedual}
    \label{prop:RS2-dual}
    If $1\leq k \leq n-2$ then
    \[ (\orthant{n}{k})^* = \left\{\diag(Y):\;\; Y \psd 0,\;\; V_n^TYV_n \in (\psdcone{n-1}{k})^*\right\}.\]
\end{propositionthreedual}
We could also obtain a dual version of Proposition~\ref{prop:btn} by directly
applying conic duality to the semidefinite representation in
Proposition~\ref{prop:btn}.  This would involve dualizing the semidefinite
representation of $\SH_n$. Instead we give another, perhaps simpler,
representation of $(\psdcone{n}{k})^*$ in terms of $(\orthant{n}{k})^*$ that is 
not obtained by directly applying conic duality to Proposition~\ref{prop:btn}. 
\begin{propositiononedual}
    \label{prop:btn-dual}
    If $(\orthant{n}{k})^*$ has a semidefinite representation of size $m$, then 
    $(\psdcone{n}{k})^*$ has a semidefinite representation of size $m+O(n^2)$ given by
    \begin{equation}
        \label{eq:dualsh} 
        (\psdcone{n}{k})^* = \left\{W\in \Sym^n: \exists y\in \R^{n}\;\;\text{s.t.}\;\;
                y\in (\orthant{n}{k})^*,\;\; (W,y)\in \SH_n\right\}.
        \end{equation}
\end{propositiononedual}
Recall that Proposition~\ref{prop:btn} holds because $\psdcone{n}{k}$ is
invariant under orthogonal conjugation and $\orthant{n}{k}$ is the diagonal
slice of $\psdcone{n}{k}$. While it is immediate that $(\psdcone{n}{k})^*$ is
also orthogonally invariant, it is a less obvious result that the diagonal
slice of $(\psdcone{n}{k})^*$ is $(\orthant{n}{k})^*$. We prove this in
Section~\ref{sec:btn}.

The recursions underlying the derivative-based and polar derivative-based
representations of $(\psdcone{n}{k})^*$ then take the form
\begin{equation}
    \label{eq:dual-td} (\psdcone{n}{k})^* \leftarrow (\orthant{n}{k})^* 
    \leftarrow (\psdcone{n-1}{k-1})^*\leftarrow  \cdots \leftarrow
    (\orthant{n-k+1}{1})^* \leftarrow (\psdcone{n-k}{0})^*
\end{equation}
and, respectively,
\begin{equation}
    \label{eq:dual-bu}
    (\psdcone{n}{k})^* \leftarrow (\orthant{n}{k})^* \leftarrow
    (\psdcone{n-1}{k})^* \leftarrow \cdots \leftarrow (\orthant{k+2}{k})^* 
    \leftarrow (\psdcone{k+1}{k})^*.
\end{equation}
Note that for the dual derivative-based representation, the base case is
$(\psdcone{n-k}{0})^* = \Sym_+^{n-k}$ (since the positive semidefinite cone is
self dual).  For the dual polar derivative-based representation the base case
is $(\psdcone{k+1}{k})^* = \{tI_{k+1}: t\geq 0\}$, the ray generated by the
identity matrix in $\Sym^{k+1}$.

\subsection{Derivative relaxations of spectrahedral cones}
\label{sec:spec}
So far we have focused on the derivative relaxations of the positive semidefinite 
cone. It turns out that the derivative relaxations of spectrahedral cones 
are just slices of the associated derivative relaxations of the positive semidefinite
cone. 
\begin{proposition}
    Suppose $p(x) = \det(\sum_{i=1}^{n}A_ix_i)$ where the $A_i$ are $m\times m$ symmetric matrices
    and $e\in \R^n$ is such that $\sum_{i=1}^{n}A_ie_i = B$ is positive definite. 
    Then for $k=0,1,\ldots,m-1$,
    \[ \Lambda^{(k)}_{+}(p,e) = \left\{x\in \R^n: \sum_{i=1}^{n}B^{-1/2}A_iB^{-1/2}x_i \in \psdcone{m}{k}\right\}.\]
\end{proposition}
\begin{proof}
    Let $\Amap(x) = \sum_{i=1}^{n}B^{-1/2}A_iB^{-1/2}x_i$. Then $\Amap(e) = I$ and for all $x\in \R^n$ and
    all $t\in \R$
    \[ p(x+te) = \det(B)\det(\Amap(x+te)) = \det(B)\det(\Amap(x)+tI).\]
    This implies that all the derivatives of $p$ in the direction $e$ are exactly the same
    as the corresponding derivatives of $\det(B)\det(X)$ in the direction $I$ evaluated
    at $X = \Amap(x)$. Since $\det(B) > 0$, it follows that for $k=0,1,\ldots,m-1$,
    $x\in \Lambda_+^{(k)}(p,e)$ if and only if $\Amap(x)\in \psdcone{m}{k}$. 
\end{proof}

We conclude this section with an example of these constructions.
\begin{example}[Derivative relaxations of a $3$-ellipse]
    Given foci $(0,0),(0,4)$ and $(3,0)$ in the plane, the $3$-ellipse
    consisting of points such that the sum of distances to the foci equals $8$
    is shown in Figure~\ref{fig:3ellipse}.  This is one connected component of
    the real algebraic curve of degree $8$ given by $\{(x,y)\in \R^2: \det
        \E(x,y,1)=0\}$ where $\E$ is defined in \eqref{eq:3ellipse} (see Nie et
    al.~\cite{nie2008semidefinite}).  The region enclosed by this $3$-ellipse
    is the $z=1$ slice of the spectrahedral cone defined by $\E(x,y,z)\psd 0$
    where
    \begin{equation}
    \label{eq:3ellipse}
    \E(x,y,z)=\begin{bmatrix} 
           5z+3x & y & y-4z & 0 & y & 0 & 0 & 0\\
            y & 5z+x & 0 & y-4z & 0 & y & 0 & 0\\
            y-4z & 0 & 5z + x & y & 0 & 0 & y & 0\\
            0 & y-4z & y & 5z-x & 0 & 0 & 0 & y\\
            y & 0 & 0 & 0 & 11z+x & y & y-4z & 0\\
            0 & y & 0 & 0 & y & 11z-x & 0 & y-4z\\
            0 & 0 & y & 0 & y-4z & 0 & 11z-x & y\\
        0 & 0 & 0 & y & 0 & y-4z & y & 11z-3x\end{bmatrix}.
    \end{equation} 
    Note that $\E(0,0,1) \pd 0$ and so $e=(0,0,1)$ is a direction of
    hyperbolicity for $p(x,y,z) = \det \E(x,y,z)$.  The left of
    Figure~\ref{fig:3ellipse} shows the $z=1$ slice of the cone
    $\Lambda_+(p,e)$ and its first three derivative relaxations
    $\Lambda_+^{(1)}(p,e),\Lambda_+^{(2)}(p,e)$, and $\Lambda_+^{(3)}(p,e)$.
    The right of Figure~\ref{fig:3ellipse} shows the $z=1$ slice of the cones
    $(\Lambda_+(p,e))^*, (\Lambda^{(1)}_+(p,e))^*, (\Lambda^{(2)}_+(p,e))^*$,
    and $(\Lambda_{+}^{(3)}(p,e))^*$. All of these convex bodies were plotted
    by computing 200 points on their respective boundaries by optimizing 200
    different linear functionals over them.  We performed the optimization by
    modeling our semidefinite representations of these cones in YALMIP
    \cite{lofberg2004yalmip} which numerically solved the corresponding
    semidefinite program using SDPT3 \cite{toh1999sdpt3}.
\end{example}
\begin{figure}
    \begin{center}
\includegraphics[width = 0.4\linewidth]{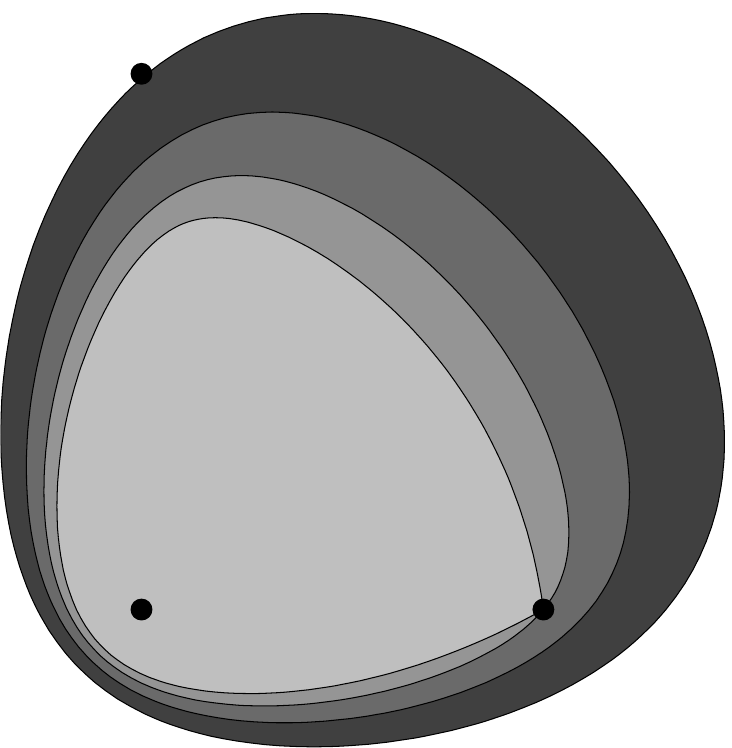}\hspace{1cm}
\raisebox{0.35cm}{\includegraphics[width=0.3\linewidth]{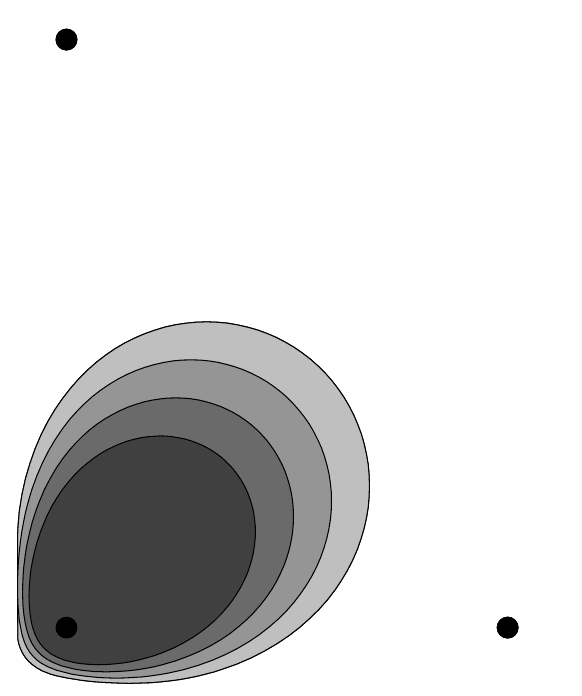}}
    \end{center}
    \caption{\label{fig:3ellipse} On the left, the inner region is the
        $3$-ellipse consisting of points with sum-of-distances to
        $(0,0),(0,4)$, and $(3,0)$ equal to $8$, i.e.~the $z=1$ slice of the
        spectrahedral cone defined by \eqref{eq:3ellipse}.  The outer three
        regions are the $z=1$ slices of the first three derivative relaxations
        of this spectrahedral cone in the direction $(0,0,1)$. On the right are
        the $z=1$ slices of the dual cones of the cones shown on the left, with
        dual pairs having the same shading.}
\end{figure}

\section{The derivative-based and polar derivative-based recursive constructions}
\label{sec:main-pfs}
In this section we prove Proposition~\ref{prop:RS1} which relates
$\orthant{n}{k}$ and $\psdcone{n-1}{k-1}$ as well as Proposition~\ref{prop:RS2}
which relates $\orthant{n}{k}$ and $\psdcone{n-1}{k}$. These relationships are
the geometric consequences of polynomial identities between elementary
symmetric polynomials and determinants.

Specifically the proof of Proposition~\ref{prop:RS1} makes use of a
determinantal representation (Equation~\eqref{eq:cosw} in
Section~\ref{sec:RS1-pfs}) of the derivative
\begin{equation}
    \label{eq:derivexpansion}
    \textstyle{\left.\frac{\partial}{\partial t}\elem_n(sx+t\ones_n)\right|_{s=1}} = 
         \left[1 \cdot \elem_{n-1}(x) + \cdots  +
             (n-1)\cdot \elem_1(x)t^{n-2} + n\cdot t^{n-1}\right].
     \end{equation}
     (Note that $s$ plays no role in~\eqref{eq:derivexpansion}, we include it
     to highlight the relationship with~\eqref{eq:polarexpansion}.) Similarly
     the proof of Proposition~\ref{prop:RS2} relies on a determinantal
     expression (Equation~\eqref{eq:buid} in Section~\ref{sec:RS2-pfs}) for the
     polar derivative
     \begin{equation}
         \textstyle{\left.\frac{\partial}{\partial s}\elem_n(sx+t\ones_n)\right|_{s=1}} = 
          \left[n\cdot \elem_n(x) + (n-1)\cdot \elem_{n-1}(x)t + 
              \cdots + 1\cdot \elem_1(x)t^{n-1}\right].\label{eq:polarexpansion}
      \end{equation}
      This explains why we call one the \emph{derivative-based representation},
      and the other the \emph{polar derivative-based representation}.

\subsection{The derivative-based recursion: relating $\orthant{n}{k}$ and $\psdcone{n-1}{k-1}$}
\label{sec:RS1-pfs}
Let $V_n$ denote an (arbitrary) $n\times (n-1)$ matrix satisfying $V_n^TV_n =
I_{n-1}$ and $V_n^T\ones_n = 0$.  Our results in this section and the next stem
from the following identity. 
 \begin{lemma}
     \label{lem:main-id}
     For all $x\in \R^n$ and all $t\in \R$, 
   \begin{equation}
       \label{eq:cosw} 
       \textstyle{\left.\frac{\partial}{\partial t}\elem_n(sx+t\ones_n)\right|_{s=1}} = 
       \elem_{n-1}(x+t\ones_n) = n\det(V_n^T\diag(x)V_n+tI_{n-1}).
   \end{equation}
    \end{lemma}
   This is a special case of an identity established by Choe et
   al.~\cite[Corollary 8.2]{choe2004homogeneous} and is closely related to
   Sanyal's result~\cite[Theorem 1.1]{sanyal2013derivative}. The proof of Choe
   et al.~uses the Cauchy-Binet identity. Here we provide an alternative proof.
   \begin{proof}
       The polynomial $\elem_{n-1}(x_1,x_2,\ldots,x_n)$ is characterized by
       satisfying $\elem_{n-1}(\ones_n) = n$,  and by being symmetric,
       homogeneous of degree $n-1$ and of degree one in each of the $x_i$. We
       show, below, that $n\det(V_n^T\diag(x)V_n)$ also has these properties
       and so that $\elem_{n-1}(x) = n\det(V_n^T\diag(x)V_n)$. The stated
       result then follows because $V_n^TV_n = I_{n-1}$ implies
        \[ \elem_{n-1}(x+t\ones_n) = n\det(V_n^T\diag(x+t\ones_n)V_n) = 
            n\det(V_n^T\diag(x)V_n + tI_{n-1}).\]
        Now, it is clear that $\det(V_n^T\diag(x)V_n)$ is homogeneous of degree
        $n-1$ and that 
        \[ n\det(V_n^T\diag(\ones_n)V_n) = n\det(I_{n-1}) = n.\] 
        It remains to establish that $\det(V_n^T\diag(x)V_n)$ is symmetric and
        of degree one in each of the $x_i$.  To do so we repeatedly use the
        fact that if $V_n$ and $U_n$ both have orthonormal columns that span
        the orthogonal complement of $\ones_n$ then $\det(V_n^T\diag(x)V_n) =
        \det(U_n^T\diag(x)U_n)$.
        
        The polynomial $\det(V_n^T\diag(x)V_n)$ is symmetric because for any
        $n\times n$ permutation matrix $P$ the columns of $V_n$ and $P V_n$
        respectively are both orthonormal and each spans the orthogonal
        complement of $\ones_n$ (because $P\ones_n = \ones_n$). Hence
        \[ \det(V_n^T\diag(Px)V_n) = \det((PV_n)^T\diag(x)(P V_n)) = \det(V_n^T\diag(x)V_n).\]
        We finally show that $\det(V_n^T\diag(x)V_n)$ is of degree one in each
        $x_i$ by a convenient choice of $V_n$. For any $i$, we can always
        choose $V_n$ to be of the form
        \[ V_n^T = \begin{bmatrix} v_1 & \cdots v_{i-1} & \sqrt{\frac{n-1}{n}}e_i 
                &  v_{i+1} & \cdots&v_n\end{bmatrix}\]
        where $e_i$ is the $i$th standard basis vector in $\R^{n-1}$. Then 
        \[ \det(V_n^T\diag(x)V_n) = \det\left(x_i\left(\textstyle{\frac{n-1}{n}}\right)e_ie_i^T + 
                \textstyle{\sum_{j\neq i}}x_j v_jv_j^T\right)\]
        which is of degree one in $x_i$ by the linearity of the determinant in
        its $i$th column.
    \end{proof}
  As observed by Sanyal, such a determinantal identity for $\elem_{n-1}(x)$
  establishes that $\orthant{n}{1}$ is a slice of $\Sym_+^{n-1} =
  \psdcone{n-1}{1-1}$. We now have two expressions for the derivative
  $\left.\frac{\partial}{\partial t}\elem_n(sx+t\ones_n)\right|_{s=1}$, one
  from the definition~\eqref{eq:derivexpansion} and one from~\eqref{eq:cosw}.  Comparing them allows us
  to deduce Proposition~\ref{prop:RS1}, that $\orthant{n}{k}$ is a slice of
  $\psdcone{n-1}{k-1}$ for all $1\leq k \leq n-1$.
   \begin{proof}[of Proposition~\ref{prop:RS1}]  
       From~\eqref{eq:derivexpansion} and~\eqref{eq:cosw} we see that 
       \begin{align*}
           \textstyle{\left.\frac{\partial}{\partial t}\elem_n(sx+t\ones_n)\right|_{s=1}} & = 
         \left[1 \cdot \elem_{n-1}(x) + \cdots  +
             (n-1)\cdot \elem_1(x)t^{n-2} + n\cdot t^{n-1}\right]\\
         & = n\left[\melem_{n-1}(V_n^T\diag(x)V_n) + \cdots + \melem_1(V_n^T\diag(x)V_n)t^{n-2} + t^{n-1}\right].
       \end{align*}
       Comparing coefficients of powers of $t$ we see that for
       $i=0,1,\ldots,n-1$
       \[ n\melem_{(n-1)-(i-1)}(V_n^T\diag(x)V_n) = (n-i)\elem_{n-i}(x).\]
       Hence for $k=1,2,\ldots,n-1$, $x\in \orthant{n}{k}$ if and only if
       $V_n^T\diag(x)V_n\in \psdcone{n-1}{k-1}$.
   \end{proof} 

    \subsection{The polar derivative-based recursion: relating $\orthant{n}{k}$ and $\psdcone{n-1}{k}$}
    \label{sec:RS2-pfs}
In this section we relate $\orthant{n}{k}$ with $\psdcone{n-1}{k}$, eventually
proving Proposition~\ref{prop:RS2}.  Our argument follows a pattern similar to
the previous section. First we give a determinantal expression for the polar
derivative $\left.\frac{\partial}{\partial s}\elem_{n}(sx+t\ones_n)\right|_{s=1}$, 
and then interpret it geometrically.

While our approach here is closely related to the approach of the previous
section, things are a little more complicated. This is not surprising because
our construction aims to express $\orthant{n}{k}$, which has an algebraic
boundary of degree $n-k$, in terms of $\psdcone{n-1}{k}$, which has an
algebraic boundary of \emph{smaller} degree, $n-k-1$. Hence it is not possible
for $\orthant{n}{k}$ simply to be a slice of $\psdcone{n-1}{k}$.

\paragraph{Block matrix notation:}
Let $\hat{\ones}_n = \ones_{n}/\sqrt{n}$ and define $Q_n = \begin{bmatrix} V_n
    & \hat{\ones}_n\end{bmatrix}$ noting that $Q_n$ is orthogonal. It is
convenient to introduce the block matrix
    \begin{equation}
        \label{eq:Mx}
        M(x) := Q_n^T\diag(x)Q_n = \begin{bmatrix} V_n^T\diag(x)V_n & 
            V_n^T\diag(x)\hat{\ones}_n\\
            \hat{\ones}_n^T\diag(x)V_n & 
            \hat{\ones}_n^T\diag(x)\hat{\ones}_n\end{bmatrix} 
        =:\begin{bmatrix} M_{11}(x) & M_{12}(x)\\
            M_{12}(x)^T & M_{22}(x)\end{bmatrix}
    \end{equation}
which reflects the fact that it is natural to work in coordinates that are
adapted to the symmetry of the problem. (Indeed $\hat{\ones}_n$ and the columns
of $V_n$ each span invariant subspaces for the permutation action on the
coordinates of $\R^n$.)

\paragraph{Schur complements:}
In this section our results are expressed naturally in term of the \emph{Schur complement}
$ (M/M_{22})(x) := M_{11}(x) - M_{12}(x)M_{22}(x)^{-1}M_{12}(x)^T$
which is well defined whenever $\elem_1(x) = nM_{22}(x) \neq 0$.  The following
lemma summarizes the main properties of the Schur complement that we use.
\begin{lemma}
    \label{lem:SCprop}
    If $M = \left[\begin{smallmatrix} M_{11}& M_{12}\\M_{12}^T &
            M_{22}\end{smallmatrix}\right]$ is a partitioned symmetric matrix
    with non-zero scalar $M_{22}$ and $M/M_{22} := M_{11} -
    M_{12}M_{22}^{-1}M_{12}^T$ then
        \begin{equation}
            \begin{bmatrix} M_{11} & M_{12}\\
                M_{12}^T & M_{22}\end{bmatrix}  = 
            \begin{bmatrix} I_{n-1} & M_{12}M_{22}^{-1}\\
                0 & I_{1}\end{bmatrix}
            \begin{bmatrix} M/M_{22} & 0\\
                0 & M_{22}\end{bmatrix}
            \begin{bmatrix} I_{n-1} & 0\\M_{22}^{-1}M_{12}^T & I_1\end{bmatrix}.
            \label{eq:blockfactorization}
        \end{equation}
    This factorization immediately implies the following properties.
    \begin{itemize}
        \item If $M$ is invertible then the $(1,1)$ block of $M^{-1}$ is given
            by $[M^{-1}]_{11} = (M/M_{22})^{-1}$.
        \item If $M_{22}>0$ then 
            \[ M \psd 0 \Longleftrightarrow M/M_{22} \psd 0.\]
    \end{itemize} 
\end{lemma}
We now establish our determinantal expression for the polar derivative. 
\begin{lemma}
    \label{lem:polar}
    If $\elem_1(x)  = nM_{22}(x) \neq 0$ then
    \begin{equation}
        \label{eq:buid}
        \textstyle{\left.\frac{\partial}{\partial s} \elem_n(sx+t\ones_n)\right|_{s=1}} 
        = \elem_1(x)\det((M/M_{22})(x) + tI_{n-1}).
    \end{equation}
\end{lemma}
\begin{proof}
    First assume $x_i \neq 0$ for $i=1,2,\ldots,n$. If $x\in \R^n$ let $x^{-1}$
    denote its entry-wise inverse.  Exploiting our determinantal expression for
    the derivative we see that 
    \begin{align}
        \textstyle{\frac{\partial}{\partial s}\elem_n(sx+t\ones_n)} & =
        \textstyle{\elem_n(x)\frac{\partial}{\partial s}\elem_{n}(s\ones_n+tx^{-1})}\nonumber\\
        % & = \elem_n(x)\elem_{n-1}(s\ones_n+tx^{-1})\nonumber\\
        & \stackrel{*}{=} \elem_n(x)\,n\,\det(V_n^T\diag(tx^{-1}+s\ones_n)V_n)\nonumber\\
        & = \elem_n(x)\,n\,\det(V_n^{T}\diag(x^{-1})V_n)\det(tI_{n-1} + s(V_n^T\diag(x^{-1})V_n)^{-1})\nonumber\\
        & \stackrel{*}{=} \elem_n(x)\elem_{n-1}(x^{-1})\det(tI_{n-1} + s(V_n^{T}\diag(x^{-1})V_n)^{-1})\nonumber\\
        & = \elem_1(x)\det(tI_{n-1} + s(V_n^T\diag(x^{-1})V_n)^{-1})\label{eq:idsc}
    \end{align}
    where the equalities marked with an asterisk are due to~\eqref{eq:cosw}.
    Since $Q_n$ is orthogonal
    $M(x)^{-1} = (Q_n^T\diag(x)Q_n)^{-1} = Q_n^T\diag(x^{-1})Q_n  = M(x^{-1})$. 
    Hence using a property of the Schur complement from Lemma~\ref{lem:SCprop}
    we see that  
    \[ (V_n^T\diag(x^{-1})V_n)^{-1} = [M(x^{-1})]_{11}^{-1} = 
        [M(x)^{-1}]_{11}^{-1} = (M/M_{22})(x).\]
    Substituting this into~\eqref{eq:idsc} establishes the stated identity,
    which, by continuity, is valid for all $x$ such that $\elem_1(x) = nM_{22}(x) \neq
    0$.
    \end{proof}
    We now have two expressions for the polar derivative,
    namely~\eqref{eq:polarexpansion} and~\eqref{eq:buid}. One comes from the
    definition of polar derivative, the other from the determinantal
    representation of Lemma~\ref{lem:polar}. Expanding each and equating
    coefficients gives the following identities.
    \begin{lemma}
        \label{lem:polarids}
        Let $x\in \R^n$ be such that $\elem_1(x) = n M_{22}(x) \neq 0$. Then
        for $k=0,1,2,\ldots,n-1$
        \[ \elem_1(x)\melem_{n-1-k}((M/M_{22})(x)) = (n-k)\elem_{n-k}(x).\]
    \end{lemma}
    \begin{proof}
        Expanding the polar derivative two ways (from Lemma~\ref{lem:polar}
        and~\eqref{eq:polarexpansion}) we obtain
        \begin{align*}
         \textstyle{\left.\frac{\partial}{\partial s}\elem_n(sx+t\ones_n)\right|_{s=1}} & = 
          \left[n\cdot \elem_n(x) + (n-1)\cdot \elem_{n-1}(x)t + 
              \cdots + 1\cdot \elem_1(x)t^{n-1}\right]\\
          & = \elem_1(x)\left[\melem_{n-1}((M/M_{22})(x)) + \melem_{n-2}((M/M_{22})(x))t + \cdots + t^{n-1}\right].
      \end{align*}
    The result follows by equating coefficients of $t^k$.
    \end{proof}
    We are now in a position to prove the main result of this section.
    \begin{proof}[of Proposition~\ref{prop:RS2}]
        From the definition of $M(x)$ in~\eqref{eq:Mx}, observe that because
        $Q_n$ is orthogonal, the constraint $\diag(x) \psd V_n Z V_n^T$ holds
        if and only if 
        \[M(x) = Q_n^T\diag(x)Q_n \psd Q_n^T(V_n Z V_n^T)Q_n = \left[\begin{smallmatrix} Z & 0\\0 & 0\end{smallmatrix}\right].\] 
        Hence we aim to establish the following statement that is equivalent to
        Proposition~\ref{prop:RS2}
        \[ \orthant{n}{k} = \left\{x\in \R^n:\,\exists Z\in \psdcone{n-1}{k}\;\;\text{s.t.}\;\;
                M(x) \psd \begin{bmatrix} Z & 0\\0 & 0\end{bmatrix}\right\}\quad\text{for $k=1,2,\ldots,n-2$}.\]
        The arguments that follow repeatedly use the fact (from
        Lemma~\ref{lem:SCprop}) that if $\elem_1(x) = n M_{22}(x) > 0$ then
        \begin{equation}
            \label{eq:SCspec}
            M(x) \psd \begin{bmatrix} Z & 0\\0 & 0\end{bmatrix}\quad \Longleftrightarrow\quad (M/M_{22})(x) \psd Z.
        \end{equation}

        With these preliminaries established, we turn to the proof of
        Proposition~\ref{prop:RS2}. First suppose there is $Z\in
        \psdcone{n-1}{k}$ such that $M(x) - \left[\begin{smallmatrix} Z & 0\\0
                & 0\end{smallmatrix}\right]\psd 0$.  There are two cases to
        consider, depending on whether $M_{22}(x)$ is positive or zero. 
        
        Suppose we are in the case where $\elem_1(x) = nM_{22}(x) > 0$. Then
        $(M/M_{22})(x) \psd Z$, so there is some $Z' \in \Sym^{n-1}_+$ such
        that 
        \[ (M/M_{22})(x) = Z+Z' \in \psdcone{n-1}{k} + \Sym^{n-1}_+ = \psdcone{n-1}{k}\]
        where the last equality holds because $\psdcone{n-1}{k} \supset
        \Sym^{n-1}_+$. It follows that $x\in \orthant{n}{k}$ because
        $\elem_1(x) > 0$ (by assumption) and by Lemma~\ref{lem:polarids}, 
        \[i\elem_{i}(x) = \elem_1(x)E_{i-1}((M/M_{22})(x)) \geq 0\quad\text{for $i=2,3,\ldots,n-k$}.\]
        Now consider the case where $\elem_1(x) = nM_{22}(x) = 0$. Since 
        \[ \begin{bmatrix} M_{11}(x) -Z & M_{12}(x)\\M_{12}(x)^{T} & M_{22}(x)\end{bmatrix} = \begin{bmatrix}
                M_{11}(x) - Z & V_n^T x/\sqrt{n}\\
                x^TV_n/\sqrt{n} & 0\end{bmatrix} \psd 0\]
        it follows that $V_n^Tx = 0$. Since, $\hat{\ones}_n^T x =0$ we see that
        $Q_n^Tx = 0$ so $x=0\in \orthant{n}{k}$. 
        
        Consider the reverse inclusion and suppose $x\in \orthant{n}{k}$. Again
        there are two cases depending on whether $\elem_1(x)$ is positive or
        zero. If $\elem_1(x) > 0$ take $Z = (M/M_{22})(x)$.  Then,
        by~\eqref{eq:SCspec}, $M(x) \psd \left[\begin{smallmatrix} Z & 0\\0 &
                0\end{smallmatrix}\right]$. To see that $Z\in \psdcone{n-1}{k}$
        note that by Lemma~\ref{lem:polarids},  
       \[E_{i}((M/M_{22})(x)) =  (i+1)\frac{\elem_{i+1}(x)}{\elem_{1}(x)} \geq 0\quad\text{for $i=1,2,\ldots,n-1-k$}.\]
       
        If $x\in \orthant{n}{k}$ and $\elem_1(x)=0$ then we use the assumption
        that $k \leq n-2$.  Under this assumption $x\in \orthant{n}{k}\cap\{x:
            \elem_1(x) = 0\} = \{0\}$. In this case we can simply take $Z=0\in
        \psdcone{n-1}{k}$ since $M(x) = 0 \psd 0 = \left[\begin{smallmatrix} Z
                & 0\\0 & 0\end{smallmatrix}\right]$.
    \end{proof}
   
    \subsection{Dual relationships}
    \label{sec:duals}
    We conclude this section by establishing Propositions~\ref{prop:RS1-dual}
    and~\ref{prop:RS2-dual}, the dual versions of Propositions~\ref{prop:RS1}
    and~\ref{prop:RS2}. Both follow from general results about conic duality,
    such as the following rephrasing of \cite[Corollary
    16.3.2]{rockafellar1997convex}.
    \begin{lemma}
        \label{lem:dual-general}
        Suppose $K\subset \R^m$ is a closed convex cone and
        $\Amap:\R^p\rightarrow \R^m$ and $\Bmap:\R^{p}\rightarrow \R^n$ are
        linear maps. Let
        \[ C = \{\Bmap(x): \Amap(x)\in K\}\subset \R^{n}.\]
        Furthermore, assume that there is $x_0\in \R^p$ such that $\Amap(x_0)$
        is in the relative interior of $K$. Then 
        \[ C^* = \{w\in \R^{n}: \exists y\in K^*\;\;\text{s.t.}\;\;
                \Bmap^*(w) = \Amap^*(y)\}.\]
    \end{lemma}
    \begin{proof}[of Proposition~\ref{prop:RS1-dual}]
    Define $\Amap: \R^{n}\rightarrow \Sym^{n-1}$ by $\Amap(x) =
    V_n^T\diag(x)V_n$ and define $\Bmap$ to be the identity on $\R^n$. Then by
    Proposition~\ref{prop:RS1}
    \[ \orthant{n}{k} = \{\Bmap(x): \Amap(x)\in \psdcone{n-1}{k-1}\}.\] 
    Clearly $\Bmap^*$ is the identity on $\R^n$ and
    $\Amap^*:\Sym^{n-1}\rightarrow \R^n$ is given by $\Amap^*(Y) = \diag(V_n Y
    V_n^T)$. Since $\Amap(\ones_n) = I_{n-1}$ is in the interior of
    $\psdcone{n-1}{k-1}$, applying Lemma~\ref{lem:dual-general} we obtain
    \[ (\orthant{n}{k})^* = \{w\in \R^{n}: \exists Y\in (\psdcone{n-1}{k-1})^*
            \;\;\text{s.t.}\;\; w = \diag(V_n Y V_n^T).\}\]
    Eliminating $w$ gives the statement in Proposition~\ref{prop:RS1-dual}.
\end{proof}
\begin{proof}[of Proposition~\ref{prop:RS2-dual}]
    Define $\Amap: \R^{n}\times \Sym^{n-1} \rightarrow \Sym^{n}\times
    \Sym^{n-1}$ by 
    \[ \Amap(x,Z) = (\diag(x) - V_n Z V_n^T,Z)\]
    and $\Bmap:\R^{n}\times \Sym^{n-1}\rightarrow \R^n$ by $\Bmap(x,Z) = x$.
    Then by Proposition~\ref{prop:RS2} 
    \[ \orthant{n}{k} = \{\Bmap(x,Z): \Amap(x,Z) \in \Sym_+^n \times \psdcone{n-1}{k}\}.\] 
    A straightforward computation shows that $\Bmap^*:\R^n\rightarrow
    \R^n\times \Sym^{n-1}$ is given by $\Bmap^*(w) = (w,0)$. Furthermore
    $\Amap^*:\Sym^{n}\times \Sym^{n-1}$ is given by $\Amap^*(Y,W) = (\diag(Y),
    W-V_n^TYV_n)$.  Since $\Amap(2\ones_n,I_{n-1})$ is in the interior of
    $\Sym_+^n\times \psdcone{n-1}{k}$, applying Lemma~\ref{lem:dual-general} we
    obtain
    \[ (\orthant{n}{k})^* = \{w\in \R^n: \exists (Y,W)\in \Sym_+^n\times (\psdcone{n-1}{k})^*\;\;\text{s.t.}\;\;
            w = \diag(W),\;\; V_n^T Y V_n = W\}.\]
    Eliminating $W$ and $w$ gives the statement in
    Proposition~\ref{prop:RS2-dual}.
\end{proof}    

    \section{Exploiting symmetry: relating $\psdcone{n}{k}$ and $\orthant{n}{k}$ and their dual cones}
    \label{sec:btn}
    In the introduction we observed that $\psdcone{n}{k}$ is invariant under the
    action of orthogonal matrices by conjugation on $\Sym^n$ and that its
    diagonal slice is $\orthant{n}{k}$. In this section we explain how to use
    these properties to construct the semidefinite representation of
    $\psdcone{n}{k}$ in terms of $\orthant{n}{k}$ stated in
    Proposition~\ref{prop:btn}. We then discuss how the duals of these two
    cones relate. The material in this section is well known so in some places
    we give appropriate references to the literature rather than providing
    proofs.

    Let $O(n)$ denote the group of $n\times n$ orthogonal matrices. The
    \emph{Schur-Horn cone} is 
    \begin{equation}
        \label{eq:SH}
        \SH_n = \left\{(X,z):\;\;z_1\geq z_2\geq \cdots \geq z_n,\;\;X\in
            \textup{conv}_{Q\in O(n)} \{Q^T\diag(z)Q\}\right\},
    \end{equation}
    the set of pairs $(X,z)$ such that $z$ is in weakly decreasing order and
    $X$ is in the convex hull of symmetric matrices with ordered spectrum $z$.
    We call this the Schur-Horn cone because all \emph{symmetric Schur-Horn
        orbitopes} \cite{sanyal2011orbitopes} appear as slices of $\SH_n$ of
    the form $\{X: (X,z_0)\in \SH_n\}$ where $z_0$ is fixed and in weakly
    decreasing order.

    Whenever a convex subset $C\subset \Sym^n$ is invariant under orthogonal
    conjugation, i.e.~$C$ is a spectral set, we can express $C$ in terms of the
    Schur-Horn cone and the (hopefully simpler) diagonal slice of $C$ as
    follows.
    \begin{lemma}
        \label{lem:sym}
        If $C\subset\Sym^n$ is convex and invariant under orthogonal
        conjugation then 
        \[ C = \{X\in \Sym^n: \exists z\in \R^n\;\;\text{s.t.}\;\; (X,z)\in
                \SH_n,\;\;\diag(z)\in C\}.\]
    \end{lemma}
    \begin{proof}
        Suppose $X\in C$. Take $z = \lambda(X)$, the ordered vector of
        eigenvalues of $X$.  Then there is some $Q\in O(n)$ such that $X =
        Q^T\diag(\lambda(X))Q$ so $(X,\lambda(X))\in \SH_n$.  By the orthogonal
        invariance of $C$, $X\in C$ implies that $QXQ^T = \diag(\lambda(X))\in
        C$.
        
        For the reverse inclusion, suppose there is $z\in \R^n$ such that
        $(X,z)\in \SH_n$ and $\diag(z)\in C$.  Then by the orthogonal
        invariance of $C$, $Q^T\diag(z)Q\in C$ for all $Q\in O(n)$.  Since $C$
        is convex, $\textup{conv}_{Q\in O(n)}\{Q^T\diag(z)Q\} \subseteq C$.
        Hence $(X,z)\in \SH_n$ implies that 
        \[ X\in \textup{conv}_{Q\in O(n)}\{Q^T\diag(z)Q\} \subseteq C.\]
    \end{proof}
    The first statement in Proposition~\ref{prop:btn} follows from
    Lemma~\ref{lem:sym} by recalling that $\psdcone{n}{k}$ is orthogonally
    invariant and $\orthant{n}{k} = \{z\in \R^n: \diag(z)\in \psdcone{n}{k}\}$.

    Proving the remainder of Proposition~\ref{prop:btn} then reduces to
    establishing the correctness of the stated semidefinite representation of
    $\SH_n$. This can be deduced from the following two well-known results.
    \begin{lemma}
        \label{lem:maj}
        If $\lambda(X)$ is ordered so that $\lambda_1(X)\geq \cdots \geq
        \lambda_{n}(X)$ then $(X,z)\in\SH_n$ if and only if $z_1\geq z_2 \geq
        \cdots \geq z_n$, 
        \[
            \tr(X) = \sum_{i=1}^{n}\lambda_i(X) = \sum_{i=1}^{n}z_i,\quad\text{and}\quad
            \sum_{i=1}^{\ell}\lambda_i(X) \leq \sum_{i=1}^{\ell}z_i \quad\text{for $\ell=1,2,\ldots,n-1$}.\]
    \end{lemma}
    In other words $(X,z)\in \SH_n$ if and only if $z$ is weakly decreasing and
    $\lambda(X)$ is \emph{majorized} by $z$. This is discussed, for example,
    in~\cite[Corollary 3.2]{sanyal2011orbitopes}. To turn this characterization
    into a semidefinite representation, it suffices to have semidefinite
    representations of the epigraphs of the convex functions $\sumk_{\ell}(X)
    := \sum_{i=1}^{\ell} \lambda_{i}(X)$. These are given by Nesterov and
    Nemirovski in~\cite[Section 6.4.3, Example 7]{nesterov1993interior}.
    \begin{lemma}
        \label{lem:sumk}
        If $2\leq \ell \leq n-1$, the epigraph of the convex function
        $\sumk_{\ell}(X) = \sum_{i=1}^{\ell}\lambda_i(X)$ has a semidefinite
        representation of size $O(n)$ given by
        \[\{(X,t): \sumk_{\ell}(X) \leq t\} = \{(X,t): \exists s\in \R,\; Z \in \Sym^n\quad\text{s.t.}\quad
                Z \psd 0,\;\;X \nsd Z+sI,\;\;\tr(Z)+s\,\ell \leq t\}.\]
        The epigraph of $\sumk_1(X)$ has a simpler semidefinite representation
        as
        \[ \{(X,t):\sumk_1(X)\leq t\} = \{(X,t): X\nsd tI\}.\]
    \end{lemma}

    We now turn to the relationship between $(\psdcone{n}{k})^*$ and
    $(\orthant{n}{k})^*$.  Note that $(\psdcone{n}{k})^*$ is invariant under
    orthogonal conjugation. So the claim (Proposition~\ref{prop:btn-dual}) that
    \[ (\psdcone{n}{k})^* = \{Y\in \Sym^n: \exists w\in \R^n\;\;\text{s.t.}\;\;
            w\in (\orthant{n}{k})^*,\;\;(Y,w)\in \SH_n\} \]
    would follow from Lemma~\ref{lem:sym} once we know that the diagonal slice
    of $(\psdcone{n}{k})^*$ is $(\orthant{n}{k})^*$.  This is a special case of
    the following result for which we give a direct proof.
    \begin{lemma}
        \label{lem:lewis}
        Suppose $C\subset \Sym^n$ is a convex cone that is invariant under
        orthogonal conjugation. Then 
        \begin{equation}
            \{y\in \R^n:\diag(y)\in C^*\}  = \{z\in \R^n: \diag(z)\in C\}^*.
            \label{eq:dsliceprojdual}
     \end{equation}
    \end{lemma}
    Note that if $C = \psdcone{n}{k}$ then the left hand side of~\eqref{eq:dsliceprojdual} is the
    diagonal slice of $(\psdcone{n}{k})^*$ and the right hand side is
    $(\orthant{n}{k})^*$. 
    \begin{proof}
        We use a description of the orthogonal projector onto the subspace of diagonal matrices
        as an average of orthogonal conjugations
        (see, e.g.,~\cite{miranda1994group} where the idea is attributed to Olkin).  
        For every subset $I\subset\{1,2,\ldots,n\}$ let $\Delta_I$ denote the
        diagonal matrix with $[\Delta_{I}]_{ii} = 1$ if $i\in I$ and
        $[\Delta_I]_{ii} = -1$ otherwise. The $\Delta_{I}$ are all orthogonal
        and act on symmetric matrices by $X\mapsto \Delta_I X \Delta_I^T$. A
        symmetric matrix is fixed by the action of all the $\Delta_I$ if and
        only if it is diagonal. Hence $\diag(\diag(X))$, the orthogonal
        projection of a symmetric matrix $X$ onto the subspace of diagonal
        matrices (the fixed-point subspace), is given by averaging over the
        action of the $\Delta_I$, i.e.\
        \begin{equation}
            \label{eq:projform}
            \diag(\diag(X)) = \frac{1}{2^{n}}\sum_{I} \Delta_I X \Delta_I^T
        \end{equation}
        where the sum is over all $2^n$ subsets of $\{1,2,\ldots,n\}$. 
        
        We now prove that
        $\{\diag(X): X\in C\}  = \{x\in \R^{n}: \diag(x)\in C\}$.
    Observe that the diagonal slice of $C$ is certainly contained 
    in the diagonal projection of $C$ giving one inclusion. For the
        other, suppose $X\in C$ is arbitrary. Since $C$ is
        orthogonally invariant, each $\Delta_I X \Delta_I^T$ is an element of
        $C$. Since $C$ is convex, it follows from~\eqref{eq:projform} that
        $\diag(\diag(X))\in C$ and is diagonal as we require.
        
        To prove~\eqref{eq:dsliceprojdual}, we apply
        Lemma~\ref{lem:dual-general} in Section~\ref{sec:duals} to obtain an 
        expression for $\{x\in \R^n:\diag(x)\in C\}^*$. For
        Lemma~\ref{lem:dual-general} to apply, we must exhibit $x_0\in \R^n$
        such that $\diag(x_0)$ is in the relative interior of $C$, denoted 
        $\textup{relint}(C)$. Let $X_0\in \textup{relint}(C)$ be arbitrary. Since $C$ is invariant
        under orthogonal conjugation, the same holds for $\textup{relint}(C)$.
        It follows that each $\Delta_I X_0 \Delta_I^T\in \textup{relint}(C)$ and by~\eqref{eq:projform}
        (and the convexity of $\textup{relint}(C)$) it follows that $\diag(\diag(X_0))\in \textup{relint}(C)$. 
        As such it suffices to take $x_0 = \diag(X_0)$.
    \end{proof}

\section{Concluding remarks}
\label{sec:conclusion}
We conclude with some comments about (the possibility of) simplifying our
representations and some open questions.

\subsection{Simplifications}
\label{sec:simplifications}
If we can simplify a representation of $\orthant{n}{k}$ or $\psdcone{n}{k}$ for
some $k=i$, that allows us to simplify the derivative-based representations for
$k\geq i$ and the polar derivative-based representations for $k \leq i$.  For
example $\orthant{n}{n-2}$ can be succinctly expressed in terms of the
second-order cone $\soc_+^{n+1} = \{x\in \R^{n+1}: (\sum_{i=1}^{n}x_i^2)^{1/2}
    \leq x_{n+1}\}$ as
\[ \orthant{n}{n-2} = \{x\in \R^{n}: (x,\elem_1(x))\in \soc_+^{n+1}\}.\]
Then we can represent $\psdcone{n}{n-2}$ in terms of the second-order cone as 
\[ \psdcone{n}{n-2} = \{Z\in \Sym^{n}:(Z,\tr(Z))\in \soc_+^{n^2+1}\}\]
because $\tr(Z) = \sum_{i=1}^n\lambda_i(Z)$ and $\sum_{i,j=1}^{n}Z_{ij}^2 =
\sum_{i=1}^{n}\lambda_i(Z)^2$.  This should be used as a base case instead of
$\psdcone{n}{n-1}$ in the polar derivative-based representations.  

As an example of this, Proposition~\ref{prop:RS2} can be used to give a concise
representation of $\orthant{n}{n-3}$ in terms of the second-order cone as 
\begin{align*}
    x\in \orthant{n}{n-3} \quad \Longleftrightarrow \quad& 
    \exists Z\in \Sym^{n-1}\;\;\text{such that}\\
& \diag(x) \psd V_n Z V_n^T \;\;\text{and}\;\;
(Z,\tr(Z))\in \soc_+^{(n-1)^2+1}.
    \end{align*}
 
\subsection{Lower bounds on the size of representations}
\label{sec:lower-bounds}
The explicit constructions given in this paper establish upper bounds on the
minimum size of semidefinite representations of $\psdcone{n}{k}$ and
$\orthant{n}{k}$.  To assess how good our representations are, it is
interesting to establish corresponding \emph{lower} bounds on the size of
semidefinite representations of $\orthant{n}{k}$ and $\psdcone{n}{k}$.  Since
$\orthant{n}{k}$ is a slice of $\psdcone{n}{k}$, any lower bound on the size of
a semidefinite representation of $\orthant{n}{k}$ also provides a lower bound
on the size of a semidefinite representation of $\psdcone{n}{k}$. Hence we
focus our discussion on $\orthant{n}{k}$. 

In the case of $\orthant{n}{n-1}$, a halfspace, the obvious
semidefinite representation of size one is clearly of minimum size. Less
trivial is the case of $\orthant{n}{0}$, the non-negative orthant. It has been
shown by Gouveia et al.~\cite[Section 5]{gouveia2013lifts} that $\R^n_+$ does
not admit a semidefinite representation of size smaller than $n$. Hence the
obvious representation of $\R^n_+$ as the restriction of $\Sym_+^n$ to the
diagonal is of minimum size. 

For each $k$, the slice of $\orthant{n}{k}$ obtained by setting the last $k$
variables to zero is $\R^{n-k}_+$.  Hence any semidefinite representation of
$\orthant{n}{k}$ has size at least $n-k$, the minimum size of a semidefinite
representation of $\R^{n-k}_+$.  This argument establishes that Sanyal's
spectrahedral representation of $\orthant{n}{1}$ of size $n-1$ is actually a
minimum size semidefinite representation of $\orthant{n}{1}$. We are not aware
of any other lower bounds on the size of semidefinite representations of the
cones $\orthant{n}{k}$ for $2\leq k \leq n-2$. 

The semidefinite representations of $\orthant{n}{k}$ given in this paper are
\emph{equivariant} in that they appropriately preserve the symmetries of
$\orthant{n}{k}$. (For a precise definition see 
\cite[Definition 4]{gouveia2013lifts}.) It is known  that symmetry matters when representing
convex sets as projections of other convex sets \cite{kaibel2010symmetry}. For
example if $p$ is a power of a prime, equivariant representations of regular
$p$-gons in $\R^2$ are necessarily much larger than their minimum-sized
non-equivariant counterparts \cite[Proposition 3]{gouveia2013lifts}.  Given
that the cones $\orthant{n}{k}$ are highly symmetric, it would also be
interesting to establish lower bounds on the size of equivariant semidefinite
representations of the derivative relaxations of the non-negative orthant.

\bibliographystyle{plain}
\bibliography{hyp_bib}
\end{document}